\documentclass[12pt, reqno]{amsart}

\usepackage{amssymb, amsmath, amsthm, amsfonts, mathrsfs,enumerate}
\usepackage[colorlinks=true]{hyperref}
\usepackage[alphabetic,lite]{amsrefs}
\usepackage{amscd}   
\usepackage{fullpage}
\usepackage{tikz-cd} 
\usepackage[all,cmtip]{xy}

\DeclareFontEncoding{OT2}{}{} 

\usepackage{color}

\newtheorem{lemma}{Lemma}

\newtheorem{cor}[lemma]{Corollary}

\newtheorem{claim*}{Claim}
\newtheorem{thm}[lemma]{Theorem}

\theoremstyle{definition}

\newtheorem{rmk}[lemma]{Remark}

\newcommand{\A}{{\mathbb A}}

\newcommand{\PP}{{\mathbb P}}

\newcommand{\F}{{\mathbb F}}
\newcommand{\Q}{{\mathbb Q}}

\newcommand{\Z}{{\mathbb Z}}

\newcommand{\kk}{{\mathbf k}}

\newcommand{\calG}{{\mathcal G}}

\newcommand{\calP}{{\mathcal P}}

\usepackage[mathscr]{euscript}

\newcommand{\scrS}{{\mathscr S}}
\newcommand{\scrT}{{\mathscr T}}

\DeclareMathOperator{\inv}{inv}

\DeclareMathOperator{\Cor}{Cor}

\DeclareMathOperator{\Norm}{Norm}
\DeclareMathOperator{\Br}{Br}

\DeclareMathOperator{\Clif}{Clif}

\newcommand{\eps}{\varepsilon}

\newcommand\numberthis{\addtocounter{equation}{1}\tag{\theequation}}

\title{Quartic del Pezzo surfaces without quadratic points}
\author{Brendan Creutz}
\author{Bianca Viray}

\address{School of Mathematics and Statistics, University of Canterbury, Private Bag 4800, Christchurch 8140, New Zealand}
\email{brendan.creutz@canterbury.ac.nz}
\urladdr{http://www.math.canterbury.ac.nz/\~{}b.creutz}

\address{University of Washington, Department of Mathematics, Box 354350, Seattle, WA 98195,~USA}
\email{bviray@uw.edu}
\urladdr{http://math.washington.edu/\~{}bviray}
\thanks{BC was supported in part by the Marsden Fund administered by Royal Society Te Ap\=arangi.}
\thanks{BV was supported in part by NSF grant DMS-2101434.}

\begin{document}

\begin{abstract}
Previous work of the authors showed that every quartic del Pezzo surface over a number field has index dividing $2$ (i.e., has a closed point of degree \(2\) modulo \(4\)), and asked whether such surfaces always have a closed point of degree $2$. We resolve this by constructing infinitely many quartic del Pezzo surfaces over $\Q$ without degree $2$ points. These are the first examples of smooth intersections of two quadrics with index strictly less than the minimal degree of a closed point.
\end{abstract}
	
\maketitle

The index of a variety over a field is the GCD of the degrees of its closed points. In general, the index can be strictly less than the minimal degree of a closed point. For example, the Fermat quartic surface $x_0^4 + x_1^4 + x_2^4 + x_3^4 = 0$ over the finite field $\F_5$ has index $1$ (since it has points over $\F_{5^2}$ and $\F_{5^3}$) but has no closed point of degree $1$. However, for some classes of varieties with simple geometry, the minimal degree of a closed point is always equal to the index. 
This holds for genus 1 curves (by the Riemann-Roch Theorem) and for hypersurfaces of degree at most \(2\) (by Springer's Theorem \cite{Springer}). It has been conjectured by Cassels and Swinnerton-Dyer that the same holds for cubic hypersurfaces. This is known over local fields, but remains open for cubic surfaces over number fields \cite{Coray}.

In this note we consider the situation for smooth intersections of two quadrics. In this case, theorems of Springer, Amer and Brumer~\cites{Springer, Amer, Brumer} combine to show that index $1$ implies the existence of a closed point of degree $1$. We prove that this does not extend to the case of larger index. Specifically, we give the first examples of smooth intersections of two quadrics for which the index is not equal to the minimum degree of a closed point.

\begin{thm}\label{thm:Intro}
	There are infinitely many pairwise non-isomorphic smooth intersections of two quadrics over $\Q$ that have index \(2\), but have no closed points of degree $2$.
\end{thm}

In earlier work~\cite{CV}, the authors proved that smooth intersections of two quadrics of dimension at least $2$ over number fields always have index dividing $2$ and, moreover, that degree $2$ points always exist in dimension at least $2$ over local fields, and in dimension at least $3$ over global fields (assuming Schinzel's hypothesis in the number field case).\footnote{Over general fields there are examples in dimension $2$ with index $4$; see \cite{CV}*{Section 7C}.} 
Colliot-Th\'el\`ene subsequently proved the result over number fields unconditionally~\cite{CT}*{Th\'eor\`eme 1.5}. In dimension $1$, it is possible to have index $1,2$ or $4$ over local and global fields, but index and minimal degree of a closed point coincide by the Riemann-Roch theorem. Thus, any smooth intersection of quadrics over a global field with index strictly less than the minimal degree of a closed point must be dimension $2$, in which case it is a quartic del Pezzo surface.

The proof of Theorem~\ref{thm:Intro} is constructive.  
In Section~\ref{sec:Local}, we introduce the family \(X_d\) of quartic del Pezzo surfaces of interest, prove some preliminary results about the local isometry type of quadric threefolds that contain them, and use~\cite{CV} to show that, under some assumptions on \(d\), the quartic del Pezzo \(X_d\) is not contained in any \(\Q\)-rational quadric threefold that contains a \(\Q\)-line. This almost immediately yields Theorem~\ref{thm:Intro} (see Section~\ref{sec:proof} for details).

In personal communication with the authors, Colliot-Th\'el\`ene sketched a proof, developed jointly with Koll\'ar, that any quartic del Pezzo surface with index \(2\) over a field of characteristic different from \(2\) has a point of degree \(2, 6\) or \(14\). In Section~\ref{sec:Degree6}, we observe that the results in~\cite{CV} imply that, assuming Schinzel's hypothesis, quartic del Pezzo surfaces over number fields always have degree \(6\) points. 

We close by giving an indication of how these examples arise.
In Section~\ref{sec:CompareToCV}, we explain how the $2$- and $3$-adic constraints in our examples ensure we are in the exceptional case where~\cite{CV}*{Theorem 1.2(5)} does not give degree $2$ points. In Section~\ref{sec:finding}, we explain how a result of Flynn~\cite{Flynn} enabled us to produce equations of surfaces with our desired properties that have small coefficients and so are amenable to exploration.

\section{A family of quartic del Pezzo surfaces}\label{sec:Local}
   For a nonzero integer $d$, let $X_d \subset \PP^4_\Q$ be the surface defined by the vanishing of the quadratic forms
\begin{align*}
	Q_0 &:= d(u_0^2 - 3u_1^2) + 6u_2^2 + 6u_2u_3 - 6u_3u_4\,,\\
	Q_\infty &:= -2du_0u_1 + u_2^2 - 2u_2u_4 - 3u_3^2 - u_4^2\,.\numberthis\label{eq:Qs}
\end{align*}
Corollary~\ref{cor:main} below shows that for suitable choices of $d$ (e.g., $d = -7$) the surface $X_d$ has no closed points of degree $2$.

\subsection{The pencil of quadrics}
The forms in~\eqref{eq:Qs} define a pencil of quadrics $\mathscr{Q} \subset \PP^4 \times \PP^1$ whose fiber over $t := (a:b) \in \PP^1$ is the quadric $3$-fold $\mathscr{Q}_t = V(bQ_0 + aQ_\infty) \subset \PP^4$; the quadrics in this pencil are exactly the quadric threefolds that contain \(X_d\). For a point $(a:1) \in \PP^1$ we will write $Q_a = Q_0 + aQ_\infty$. Let $M_0,M_\infty$ be the symmetric matrices corresponding to $Q_0,Q_\infty$ and let 
\[
	\Phi(T) := \det(M_0 + TM_\infty) = -6d^2(T^2+3)(T^3 + 3T^2 + 3T - 9)\,.
\]
Let $\scrS \subset \PP^1$ be the subscheme defined by the vanishing of $\Phi(T)$; the singular quadrics in the pencil all lie above $\scrS$. Let $\scrT$ and $\scrT'$ be the degree $2$ and degree $3$ subschemes defined by the irreducible factors of $\Phi(T)$, and let \(\zeta\) denote a root of the cubic factor of \(\Phi(T)\). The fibers of $\mathscr{Q} \to \PP^1$ above $\scrT$ and $\scrT'$ are rank $4$ quadrics with discriminants (i.e., the signed determinant of the restriction of the corresponding quadratic form to a smooth hyperplane section) given by $\eps_\scrT = [d]\in\kk(\scrT)^\times/\kk(\scrT)^{\times 2}$ and $\eps_{\scrT'} = [\zeta]\in\kk(\scrT')^\times/\kk(\scrT')^{\times 2}$ (See \cite[Section 4A]{CV}).

\subsection{Lines on the quadrics over local fields}

\begin{lemma}\label{lem:lines}
	Let $Q$ be a rank $5$ quadratic form over a field $k$ of characteristic not equal to $2$ and assume that $Q$ is an orthogonal sum of a hyperbolic plane and a rank $3$ form. The quadric $3$-fold $V(Q) \subset \PP^4_k$ contains a $k$-rational line if and only if the conic in $\PP^2_k$ defined by the rank $3$ form contains a $k$-rational point.
\end{lemma}

\begin{proof}
	We can assume $Q = u_0u_1 + q(u_2,u_3,u_4)$. The hyperplane tangent to $V(Q)$ at the $k$-point $P = (0:1:0:0:0)$ is $V(u_0)$. The intersection $V(Q,u_0)$ is a cone over the conic $V(q) \subset \PP^2_k$. Hence, if the conic contains a $k$-point, then $V(Q)$ contains a $k$-rational line. On the other hand, if $V(Q)$ contains a $k$-rational line, then $Q$ contains two hyperbolic planes~\cite{Lam}*{Chapter I, Theorem 3.4}. By Witt's Cancellation Theorem~\cite{Lam}*{Chapter I, Theorem 4.2}, the rank $3$ form must also contain a hyperbolic plane. Equivalently, this means that the conic has a $k$-point.
\end{proof}

\begin{lemma}\label{lem:Ginfty}
   Let $v$ be a place of $\Q$. The quadric $V(Q_\infty)$ contains a $\Q_v$-rational line if and only if $v \not\in \{2,3\}.$
   \end{lemma}
   \begin{proof}
   	The form $Q_\infty$ is (over $\Q$) the direct sum of a hyperbolic plane and the rank $3$ form $u_2^2 - 2u_2u_4 - 3u_3^2 - u_4^2$. The latter diagonalizes as \((u_2 - u_4)^2 - 3u_3^2 - 2(u_4)^2\) and, hence, (by definition of the Hilbert symbol) has a nontrivial solution over \(\Q_v\) if and only if \((2,3)_v=1\). One can compute that \((2,3)_v = 1\) if and only if \(v\notin \{2,3\}\). Now apply Lemma~\ref{lem:lines}.	
   \end{proof}

   \begin{lemma}\label{lem:G3}
   	For all $b\in \Z_3$, the quadric \(V(bQ_0 + Q_{\infty}) \subset \PP^4_{\Q_3}\) contains no \(\Q_3\)-rational lines.
   \end{lemma}
   \begin{proof}
   	First note that all points in $\scrS(\overline{\Q_3})$ reduce to \((0:1)\in \PP^1(\overline{\F_3})\), and thus that $(1:b) \not\in \scrS$. This means that $bQ_0+Q_\infty$ has rank $5$. In particular, it is the orthogonal sum of the 
rank \(2\) form \(d(bu_0^2 - 2u_0u_1 - 3bu_1^2)\) and a rank $3$ form which diagonalizes as
      \begin{equation}\label{eq:conic}
         -(u_4 + u_2 + 3bu_3)^2 + 2(1 + 3b)\left(u_2 + \frac{6b}{2 + 6b}u_3\right)^2 - 3\left(1 + 3b^2\cdot\frac{1 - 3b}{1 + 3b}\right)u_3^2\,.
      \end{equation}
      Since \(b\in \Z_3\), the rank \(2\) form is congruent to \(-2u_0u_1\) or \(bu_0^2 - 2u_0u_1\) modulo \(3\), so in particular is isotropic, i.e., a hyperbolic plane, over \(\Q_3\).
      %
      We may therefore apply Lemma~\ref{lem:lines} to the quadric \(V(bQ_0 + Q_{\infty})\) over $\Q_3$. This gives that \(V(bQ_0 + Q_{\infty})\) contains a line over $\Q_3$ if and only if the conic defined by~\eqref{eq:conic} contains a $\Q_3$-point. For any \(b\in \Z_3\), we have that \(1 + 3b\) and \(1 + 3b^2\cdot\frac{1 - 3b}{1 + 3b}\) are squares in \(\Z_3\). So the Hilbert symbol associated to this conic is $(2,-3)_3$, which is nontrivial.
  \end{proof}

  \begin{rmk}
  	For $v|6$, one can find quadrics in the pencil containing a $\Q_v$-line, as is implied by the local case in~\cite{CV}*{Theorem 1.2}. As the existence of these quadrics is not needed for the proof of Theorem~\ref{thm:Intro}, we omit the details.
  \end{rmk}

\subsection{Lines on the quadrics over $\Q$}\label{sec:Adelic}

\begin{thm}\label{thm:d}
	Suppose the integer $d$ satisfies the following:
\begin{enumerate}[(i)]
	\item\label{it:1} $d \equiv 5 \bmod 12$, and
	\item\label{it:4} for all primes $v$ dividing $d$, we have that $\zeta$ is a square in $\Q_v\otimes \Q(\zeta)$.
\end{enumerate}
Then no quadric in $\PP^4_\Q$ containing $X_d$ contains a line defined over $\Q$.
\end{thm}

\begin{proof}
Any quadric in $\PP^4_\Q$ containing $X_d$ must lie in the pencil spanned by $Q_0$ and $Q_\infty$. Let $\pi\colon\calG \to \PP^1$ be the scheme such that $\calG_t := \pi^{-1}(t)$ parameterizes the lines on the quadric $3$-fold $V(Q_t) \subset \PP^4$. The generic fiber of $\calG \to \PP^1$ is a Brauer-Severi variety and $\calG$ is smooth over $\Q$ by~\cite{CV}*{Proposition 4.2}. We must show that $\calG(\Q)= \emptyset$. We will do this by showing that there is a Brauer-Manin obstruction to rational points on $\calG$ (We refer to~\cite{CTS-BrauerBook}*{Section 13.3} for an introduction to the Brauer-Manin obstruction). 

By \cite{CV}*{Proposition 5.1} we have that $\Br(\calG)$ is generated modulo constant algebras by 
\begin{equation}\label{eq:beta}
\beta := 
\pi^*(d,T^2 + 3) = \pi^*{\Cor_{\Q(\sqrt{-3})/\Q}(d, T - \sqrt{-3})}\,,
\end{equation}
where $(d,T^2+3)$ denotes the quaternion algebra in $\Br(\kk(\PP^1)) = \Br(\Q(T))$. By \cite{CV}*{Proposition 5.1} and Lemma~\ref{lem:Ginfty}, $\beta$ may also be expressed as
\begin{equation}\label{eq:beta2}
	\beta = \pi^*\Cor_{\Q(\zeta)/\Q}(\zeta, T - \zeta) + [\calG_\infty] = \pi^*\left(\Cor_{\Q(\zeta)/\Q}(\zeta, T - \zeta) + (2,3)\right)\,.
\end{equation}

Let $S = \{ 2,3,5,\infty \} \cup \{ v \;:\; v \mid d \}$. The discriminant of $\Phi(T)$ is an $S$-unit, so the equations defining the pencil give a smooth model of $\calG$ over $\Z_S$. Then, for $v \not\in S$, the evaluation map $\inv_v\circ \beta \colon \calG(\Q_v) \to \Q/\Z$ is constant by the results of \cite{GoodReductionBM}. By evaluating at points of \(\calG_{\infty}\) (which exist by Lemma~\ref{lem:Ginfty}), we deduce that \(\inv_v\circ \beta\) is identically \(0\) for \(v\notin S\). 

Now consider \(v\in S\). If \(d\) is a square in \(\Q(\sqrt{-3})\otimes_{\Q} \Q_v\), then~\eqref{eq:beta} gives that $\inv_v\circ \beta = 0$. This holds for $v \in \{2,\infty\}$, since \(d\equiv 1 \bmod 4\) by~\eqref{it:1}. If \(\zeta\) is a square in \(\Q(\zeta)\otimes_{\Q} \Q_v\), then~\eqref{eq:beta2} gives that $\inv_v\circ \beta = \inv_v(2,3)$. This holds for $v = 5$ by an easy computation and for $v \mid d$ by Condition~\eqref{it:4}. Moreover, $d$ is prime to $6$ by~\eqref{it:1}, so $\inv_v\circ \beta = \inv_v(2,3) = 0$ for all $v \mid 5d$.
Finally, consider $v = 3$. By Lemma~\ref{lem:G3}, if $P \in \calG(\Q_3)$, then $\pi(P) = (a:1)$ for some \(a\in3\Z_3\). Since $d \equiv 2 \bmod 3$ by~\eqref{it:1}, we then have
\(
	\inv_3(\beta(P)) = \inv_3(d,a^2 + 3) = \frac12\,.
\)

Thus, for any $(P_v) \in \prod_v \calG(\Q_v) = \calG(\A_\Q)$, we have
\(
	\sum_v \inv_v(\beta(P_v)) = \inv_3(\beta(P_v)) = \frac12\,.
\)
This means there is a Brauer-Manin obstruction to the existence of $\Q$-points on $\calG$. Indeed, any $\Q$-point $P \in \calG(\Q) \subset \calG(\A_\Q)$ would have $\beta(P) \in \Br(\Q)$ so $\sum_v\inv_v(\beta(P)) = 0 \neq \frac12$.
   \end{proof}

\subsection{The Proof of Theorem~\ref{thm:Intro}}\label{sec:proof}
\begin{cor}\label{cor:main}
	If $d$ satisfies the conditions of Theorem~\ref{thm:d}, then the surface $X_d$ has no closed points of degree dividing $2$.
\end{cor}
\begin{proof}
	If $X_d$ contains a degree $2$ point, then the line in $\PP^4_\Q$ between that point and its conjugate is defined over $\Q$ and lies on some quadric in the pencil. But by Theorem~\ref{thm:d}, no quadric in the pencil contains a $\Q$-line, so $X_d$ has no degree $2$ points. A similar argument shows there are no degree $1$ points; see \cite{CV}*{Proposition 4.4} for further details. Alternatively, one can simply note that $X_d(\Q_3) = \emptyset$ to deduce that $X_d(\Q) = \emptyset$.
\end{proof}
\begin{lemma}\label{lem:Nonisomorphic}
   Let \(d, d'\) be two nonzero integers. If \(dd'\notin \Q(\sqrt{-3})^{\times2}\), then \(X_d\not\simeq X_{d'}\).
\end{lemma}
\begin{proof}
   The anticanonical embedding of a quartic del Pezzo surface \(X\) presents \(X\) as an intersection of quadrics in \(\PP^4\).  Thus, \(X_d\simeq X_{d'}\) if and only if there is a change of coordinates of \(\PP^4\) that gives the isomorphism. A linear automorphism of \(\PP^4\) will preserve the field of definition, the rank and the discriminant of any quadric hypersurface. Note that each of \(X_d\) and \(X_{d'}\) is contained in exactly two rank \(4\) quadrics defined over \(\Q(\sqrt{-3})\) with discriminant \(d\Q(\sqrt{-3})^{\times2}\) and \(d'\Q(\sqrt{-3})^{\times2}\) respectively. Thus, if \(dd'\notin \Q(\sqrt{-3})^{\times2}\), then \(X_d\not\simeq X_{d'}\).
\end{proof}

\begin{proof}[Proof of Theorem~\ref{thm:Intro}]
   Corollary~\ref{cor:main} and Lemma~\ref{lem:Nonisomorphic} imply that each squarefree \(d\) that satisfies the conditions of Theorem~\ref{thm:d} gives a quartic del Pezzo surface with no points of degree \(1\) or \(2\) and that these are all pairwise nonisomorphic. Thus, by the Springer and Amer-Brumer theorem, the index of \(X_d\) is greater than \(1\). Hence, by~\cite{CV}*{Theorem 1.1}, \(X_d\) must have index equal to \(2\). It remains to show that there are infinitely many squarefree integers \(d\) that satisfy the conditions of Theorem~\ref{thm:d}. 
   
   Consider the Galois closure \(L\) of \(\Q(\sqrt{-1}, {\sqrt{-3}},\sqrt{\zeta})\). By the Chebotarev Density Theorem, there is an infinite set of primes \(\calP\) that split completely in \(L\).  We claim that for all \(p\in \calP\), we have that \(p\equiv 1 \bmod 12\) and that \(\zeta\) is a square in \(\Q_p\otimes \Q(\zeta)\).  This claim together with the fact that \(\zeta\) is a square in \(\Q_7\otimes \Q(\zeta)\) then implies that for any finite subset \(S\subset \calP\), the integer \(d = -7\prod_{p\in S}p\) will satisfy the conditions of Theorem~\ref{thm:d}.
   
   Now let us prove the claim. If \(p\) splits completely in \(L\), then it splits completely in \(\Q(\sqrt{-1})\), in \(\Q({\sqrt{-3}})\) and in \(\Q(\sqrt{\zeta})\).  Thus, \(-1, -3\in \Q_p^{\times 2}\), which, by quadratic reciprocity, implies that \(p\equiv 1 \bmod 4\) and \(p\equiv 1 \bmod 3\). Further, since \(p\) splits completely in the Galois closure of \(\Q(\sqrt{\zeta})\), the cubic equation \(T^3 + 3T^2 + 3T - 9\) has three roots over \(\Q_p\), and, moreover, each root is a square.
\end{proof}

\section{Complements}

\subsection{Degree $6$ points on quartic del Pezzo surfaces}\label{sec:Degree6}

Every quartic del Pezzo surface over a number field has a closed point of degree equal to $2 \bmod 4$ by~\cite{CV}*{Theorem 1.1}. The proof constructs an adelic $0$-cycle of degree $1$ on the corresponding variety $\calG$ that is orthogonal to $\Br(\calG)$. A closer examination of the proof shows that one can also construct an \emph{effective} adelic $0$-cycle of degree $3$ that is orthogonal to $\Br(\calG)$. This allows us to deduce the following.

\begin{thm}\label{thm:6}
	Let $X$ be a smooth quartic del Pezzo surface over a number field. Assume Schinzel's hypothesis. Then $X$ has a closed point of degree $6$.
\end{thm}
\begin{proof}
	It will be enough to show that there is a cubic extension $K/k$ such that, for the variety $\calG/k$ parameterizing lines on quadrics in the pencil of quadrics containing $X$, we have $\calG(\A_K)^{\Br} \ne \emptyset$. Indeed, we may then apply an unpublished result of Serre (see~\citelist{\cite{CTSD94}*{Section 4, Theorem 4.2}, \cite{Serre-GalCohom}*{Chap. II, Annexe, Th\'eor\`eme 7.6}} for details)
   that, assuming Schinzel's hypothesis, the Brauer-Manin obstruction is the only obstruction to the Hasse principle for pencils of Severi-Brauer varieties. Thus we have $\calG(K) \ne \emptyset$ and, consequently, there is a closed point of degree $2$ on $X_K$~\cite{CV}*{Proposition 4.4(1)}. This gives a point of degree $3\times 2 = 6$ on $X$ over $k$.
  
   If $\calG(\A_k)^{\Br} \neq \emptyset$, then  by \cite{CV}*{Lemma 3.2} we have $\calG(\A_K)^{\Br} \ne \emptyset$ for any extension $K/k$. So we may assume that $\calG(\A_k)^{\Br} = \emptyset$. Hence, we are in the exceptional case that is not covered by~\cite{CV}*{Corollary 6.1}, which is discussed in \cite{CV}*{Remark 6.2 and Section 7A}. We have that $\Br(\calG)/\Br_0(\calG)$ is cyclic, generated by an element $\beta_\scrT \in \Br(\calG)$ of order $2$ corresponding to an irreducible degree $2$ subscheme $\scrT$ of the locus of singular quadrics in the pencil containing $X$. Moreover, the discussion in \cite{CV}*{Section 7A} gives a place $v$ of $k$ (in the set $R'_\scrT$ described there) such that $v$ does not split in $\kk(\scrT)$ and the evaluation map $\beta_\scrT:\calG(\kk(\scrT_{k_v})) \to \Br(\kk(\scrT_{k_v}))$ is nonconstant. Let $K/k$ be any cubic extension such that $K \otimes k_v \simeq k_v \times \kk(\scrT_{k_v})$ (i.e., \(K/k\) is a degree \(3\) extension where \(v\) has one degree \(1\) place lying above it and the other splitting is determined by the splitting in the quadratic extension \(\kk(\scrT)\)). Then $K$ has a place $w \mid v$ such that $K_w \simeq \kk(\scrT_{k_v})$ and for this place the evaluation map $\inv_w\circ \beta_\scrT: \calG(K_w) \to \Q/\Z$ is nonconstant. It follows that $\calG(\A_K)^{\Br} \ne \emptyset$.
\end{proof}

\begin{rmk}
	The proof above gives explicit cubic extensions $K/k$ such that $\calG(\A_K)^{\Br} \ne \emptyset$. For $X=X_d$ with $d$ as in Theorem~\ref{thm:d}, it gives that $\calG(\A_K)^{\Br} \ne \emptyset$ for the cubic extension $K = \Q(\zeta)$. For specific values of $d$ it is possible to check directly that $\calG(\Q(\zeta)) \ne \emptyset$ by searching for quadrics in the pencil which contain $K$-rational lines (which can be checked by computing the Hasse invariant of the corresponding quadratic form). Using this approach we found $K$-points on the variety $\calG$ corresponding to $X_d$ for all $d$ with $|d|<1000$ which satisfy the hypothesis of Theorem~\ref{thm:d}~\cite{code}. For example, when $d = -7$ we find that $\calG_t(\Q(\zeta)) \ne \emptyset$ for the point $t = ( 5\zeta + 3 : 7) \in \PP^1(\Q(\zeta))$.
\end{rmk}


\subsection{Comparison with \cite{CV}*{Theorem 1.2}}\label{sec:CompareToCV}

We will now show that the surfaces $X_d$ in the theorem do not satisfy the hypotheses of \cite{CV}*{Theorem 1.2(5)} (so the results of that paper imply only that \(X_d\) has a $0$-cycle of degree $2$, not necessarily a quadratic point). Specifically, we note that $X_d = V(Q_{\sqrt{-3}}) \cap V(Q_{-{\sqrt{-3}}})$ for the singular quadrics $V(Q_{\pm{\sqrt{-3}}}) \subset \PP^4_{\Q({\sqrt{-3}})}$ and we will show that there are an odd number of places of \(\Q({\sqrt{-3}})\) for which these fail to have smooth local points.

\begin{lemma}\label{lem:compare}
	Assume that \(d\equiv 5 \bmod 12\). Then, the singular quadric $V(Q_{\sqrt{-3}}) \subset \PP^4_{\Q({\sqrt{-3}})}$ has smooth points over all completions of $\Q({\sqrt{-3}})$ except at the unique prime of $\Q({\sqrt{-3}})$ above $2$ (where it does not have smooth points).
\end{lemma}	

\begin{proof}
   The field \(K= \Q({\sqrt{-3}})\) is totally imaginary, so it suffices to consider completions at nonarchimedean places \(v\).
	The quadric $V(Q_{\sqrt{-3}})$ has smooth points if and only if the rank $4$ quadratic form $Q_{\sqrt{-3}}$ is isotropic. Over a local field there is a unique anisotropic form of rank $4$, and it has square discriminant.
   We rewrite \(Q_{{\sqrt{-3}}}\) as
   \[
      d(u_0 - {\sqrt{-3}} u_1)^2 - {\sqrt{-3}}\left(u_4 -{\sqrt{-3}} u_3 + u_2\right)^2 + 2{\sqrt{-3}}\left({\sqrt{-3}} u_3 - u_2\right)^2 - 2({\sqrt{-3}} u_2)^2,
   \]
   and see that its discriminant is \(dK^{\times2}\). Thus, if \(d\not\in K_v^{\times2}\), the quadric \(V(Q_{{\sqrt{-3}}})\) has smooth \(K_v\)-points.  In particular, since \(d\equiv 2 \bmod 3\), \(V(Q_{{\sqrt{-3}}})\) has smooth \(K_v\)-points for the unique place \(v|3\).
   
   Let \(v\) be such that \(d \in K_v^{\times 2}\) and let \(L\) be the \'etale $K_v$-algebra \(K_v[z]/(z^2 - {\sqrt{-3}})\). Then for a suitable choice of linear forms \(\ell_0, \dots \ell_3\), we may write \(Q_{{\sqrt{-3}}}\) as
   \[
      \Norm_{L/K_v}(\ell_0 - {z}\ell_1) = 2\Norm_{L/K_v}(\ell_2 -{z}\ell_3).
   \]
   In particular, for $v$ such that \(d\in K_v^{\times2}\), \(V(Q_{{\sqrt{-3}}})\) has a smooth \(K_v\) point if and only if \(2 \in \Norm(L/K_v)\) or, equivalently, if and only if \({\inv_v}(2, {\sqrt{-3}})\) is trivial. We compute
   \[
      {\inv_v}(2,{\sqrt{-3}}) = \begin{cases}
         0 & v\nmid 6,\\
         \frac12 & v \mid 3.
      \end{cases}
   \]
   By Hilbert reciprocity, we deduce that \(\inv_w(2,{\sqrt{-3}}) = \frac12\) for the unique prime $w$ above $2$.
   \end{proof}

\begin{rmk}
	In the notation of \cite{CV}*{Definitions 5.7 and 5.15} the proof of the lemma shows that $C_\scrT := \Cor_{\kk(\scrT)/\Q}(\Clif(Q_\scrT))$ is equal to the class in $\Br(\Q)$ of the quaternion algebra ramified at $2$ and $3$ and that $R_\scrT = \{ 2 \}$. In particular, $R_\scrT$ has odd cardinality and we are in the exceptional case considered in \cite{CV}*{Remark 6.2 and Section 7A} where the results of that paper are not able to show $\calG(\A_\Q)^{\Br} \ne \emptyset$. 
\end{rmk}

\begin{rmk}
	For any rank $4$ quadric $V(Q) \subset \PP^4_L$ over a quadratic extension $L/k$ of a global field $k$ such that $V(Q)$ is not defined over $k$ and has smooth points over all but an odd number of completions of $L$,  \cite{CV}*{Corollary 6.3} gives that $\calG(\A_k)^{\Br} \ne \calG(\A_k)$ for the variety $\calG/k$ corresponding to pencil over $k$ containing $V(Q)$ and its Galois conjugate. To get that $\calG(\A_k)^{\Br} = \emptyset$ one must also have that the local evaluation maps are all constant. This is not the case for all surfaces \(X_d\) in our family. For example, if $\calG/\Q$ corresponds to the surface with $d = -79$, then it satisfies Lemma~\ref{lem:compare} because $-79 \equiv 5 \bmod 12$, but one can also check that the evaluation map is nonconstant at $v = 79$. We note that the second condition on $d$ in Theorem~\ref{thm:d} fails for $d = -79$.
\end{rmk}

\subsection{Finding surfaces and equations}\label{sec:finding}

We first identified candidate surfaces for the theorem by a computer search which we now describe. All computations were performed using the Magma computational algebra system \cite{BCP}. The code is publicly available~\cite{code}.

Following \cite{Flynn}, up to isomorphism any quartic del Pezzo surface \(X\) can be specified by the data $(\scrS,\eps_\scrS)$, where $\scrS \subset \A^1$ is a reduced degree $5$ subscheme and $\eps_\scrS \in \kk(\scrS)^\times$ is an element of square norm. Assuming Schinzel's hypothesis, if \(X\) has no quadratic points then $\scrS$ must contain an irreducible degree $2$ subscheme~\cite{CV}*{Theorem 1.2}. So we considered $\scrS$ defined by polynomials $f(T) = f_2(T)f_3(T)$ with the $f_i(T) \in \Z[x]$ irreducible monic polynomials of degree $i$ of discriminant up to some bound (taken from Magma's number field database). For each $f(T)$, we ran through $\eps_\scrS$ of the form $(d,\eps)\in (\Q[T]/\langle f_2(T)\rangle)^{\times} \times (\Q[T]/\langle f_3(T)\rangle)^{\times}\simeq \kk(\scrS)^{\times}$ with $d$ a nonzero squarefree integer of some bounded size and $\eps \in \Q[T]/\langle f_3(T) \rangle$ one of the finite set of square classes of $S$-units with square norm (for a fixed set $S$ of primes). The reason why we take the first component of $\eps_\scrS$ to lie in $\Q$ (rather than in $\Q[T]/\langle f_2(T)\rangle$) is that if this condition is not satisfied, then the Brauer group of the corresponding surface is trivial and there can be no Brauer-Manin obstruction to quadratic points.

For a given pair $(f(T),\eps_{\scrS})$ equations for the corresponding surface can be computed as follows. Letting $\theta$ denote the image of $T$ in $\Q[T]/\langle f(T)\rangle$, consider the equation
	\begin{equation}\label{eq:equation}
	(x_1 - \theta x_3)(x_1-\theta x_3) = \eps_{\scrS}(u_0 + u_1\theta + u_2\theta^2 + u_3\theta^3 + u_4\theta^4)^2\,.
	\end{equation}
	Multiplying this out and equating coefficients on the basis vectors $1,\theta,\dots,\theta^4$ of the $\Q$-vector space $\Q[T]/\langle f(T)\rangle$ yields five quadratic forms over $\Q$. The coefficients on $\theta^3$ and $\theta^4$ are quadratic forms involving only the $u_i$, and together they define a pencil of quadrics that has singular locus and discriminants isomorphic to $(f(T),\eps_{\scrS})$ by \cite{BBFL}*{Lemma 17}. 
	
	In the case we consider, we have that $f(T)=f_2(T)f_3(T)$ is reducible. We found that we obtain quadratic equations in block diagonal form (and with much smaller coefficients) if we express~\eqref{eq:equation} in terms of the basis $(1,0),(\theta_2,0),(0,1),(0,\theta_3),(0,\theta_3^2)$ of $\Q[\theta_2]\times \Q[\theta_3] = \Q[T]/\langle f(T) \rangle$ where $\theta_i$ denotes a root of  $f_i(T)$ (as opposed to the power basis.). 
	
	Having obtained quadratic forms defining the surface $X = X(f(T),\eps_{\scrS})$, we proceed to search for \(\Q\)-points on $\calG$ as follows. For a given $t \in \PP^1(\Q)$, we can check if $\calG_t(\Q) \ne \emptyset$ by computing the Hasse invariant of the quadratic form $Q_t$ (which reduces to 
   diagonalizing and computing Hilbert symbols). If no points are found among $t$ of small height, we turn to trying to show that $\calG(\Q) = \emptyset$. In this case we check whether the evaluation maps $\inv_v\circ \beta \colon \calG(\Q_v) \to \Q/\Z$ appear to be constant for primes $v$ of bad reduction. To do this we find those $t = (a:b)$ among our search space such that $\calG_t(\Q_v) \ne \emptyset$ (using the Hasse invariant) and for each such $t$ compute $\inv_v(\beta(\calG_t(\Q_v))) = \inv_v(d,f_2(t))$. If all $t$ considered yield the same value, we suspect that the evaluation map is constant and have determined a value in its image.
	
	Following this process we eventually found a surface (the surface $X_d$ with $d = -19$) for which $\calG$ had no $\Q$-points on fibers above points of small height and the evaluation maps all appeared to be constant, with an odd number of them nonzero. Thus, we expected that $\calG$ had a Brauer-Manin obstruction to $\Q$-points. Armed with the small height equations~\eqref{eq:Qs} for a candidate we were then able to prove (as described in the sections above) that there is in fact a Brauer-Manin obstruction, and that with minor modifications the same proof works for infinitely many other values of $d$. 


\end{document}